\documentclass[12pt, leqno]{amsart}
\usepackage[body={6in,9in},top=1in,left=1.25in, head=0.5in]{geometry}
\usepackage{amsmath,amsthm}
\usepackage{amssymb}
\usepackage{enumerate}
\usepackage{enumitem}
\usepackage{setspace}
\usepackage{mathtools}
\usepackage{soul}
\newtheorem{thm}{Theorem}[section]
\newtheorem{cor}[thm]{Corollary}

\newtheorem{prop}[thm]{Proposition}

\newtheorem{thrm}[thm]{Theorem}
\newtheorem{lm}[thm]{Lemma}
\newtheorem{rmk}[thm]{Remark}

\newtheorem{xmpl}[thm]{Example}

\theoremstyle{definition}
\newtheorem{defin}[thm]{Definition}
\newtheorem{rem}[thm]{Remark}
\newtheorem{exa}[thm]{Example}

\newcommand{\bn}{\mathbb{N}}

\newcommand{\br}{\mathbb{R}}

\newcommand{\ov}[1]{\overline{#1}}

\newcommand{\cl}{\textup{cl}}

\newcommand{\pref}[1]{\textup{(\ref{#1})}}

\begin{document}
\title[compactness and fixed point in partial metric spaces]{On compactness and fixed point theorems in partial metric spaces}

\author{Dariusz Bugajewski}
\author{Piotr Ma\'ckowiak}
\author{Ruidong Wang}

\address{D. Bugajewski, Department of Mathematics and Computer Science\\
Adam Mickiewicz University in Pozna\'n\\Umultowska 87\\
61-614 Pozna\'n, Poland}
\email{ddbb@amu.edu.pl}
\address{P.Ma\'ckowiak, Department of Mathematical Economics\\
Pozna\'n University of Economics and Business\\
Al.~Nie\-podleg\l{}o\'sci 10\\
61-875 Pozna\'n, Poland}
\email{piotr.mackowiak@ue.poznan.pl}
\address{R. Wang, College of Science\\
Tianjin University of Technology\\
300384, Tianjin, China}
\email{wangruidong@tjut.edu.cn}
\date{}

\subjclass[2010]{54D30, 54E35, 54E50, 54H25}
\keywords{Banach Contraction Principle, compactness, completeness, fixed point theorem, metrizability, partial metric spaces, sequential compactness}

\begin{abstract}
In this paper we examine two basic topological properties of partial metric spaces, namely compactness and completeness. Our main result claims that in these spaces compactness is equivalent to sequential compactness. We also show that Hausdorff compact partial metric spaces are metrizable. In the second part of this article we discuss the significance of bottom sets of partial metric spaces in fixed point theorems for mappings acting in these spaces.
\end{abstract}

\maketitle

\section{Introduction}

The aim of the present paper is to show that, in spite of the fact that many results are lost when passing from metric to partial metric spaces, there are some ones which are preserved by this generalization. 

It is not necessary to convince anyone how important the notion of compactness is. It is well-known that in the case of metric spaces compactness is equivalent to sequential compactness. Therefore it is natural to ask about a relation between those two notions in the case of partial metric spaces. It appears that also in the case of partial metric spaces these two notions are equivalent. That is the main result of this paper. Let us notice that, while the proof of the fact that compactness implies sequential compactness in partial metric spaces is standard and trivial, the tools we use to prove the inverse implication are much more subtle. We also show that any Hausdorff compact partial metric space is a metrizable space.

The other issues we deal with in this article are connected with the notion of completeness in partial metric spaces as well as with fixed point theorems for mappings acting in these spaces. In particular, we indicate a partial metric space and a contraction acting in it, which is not a continuous mapping (the reader interested in contractions and generalized contractions acting in partial metric spaces is referred e.g. to \cite{ASS}, \cite{ADA}, \cite{CKT}, \cite{IPR} and \cite{MD}). 

A particular attention is paid to the set of complete elements (we call it bottom set) of partial metric spaces and their role in fixed point theorems. A useful tool in these considerations is the fixed point theorem proved in the paper \cite{IPR}. Our investigations lead to the conclusion that the assumption of the non-emptiness of a bottom set appearing in fixed point theorems seems to be a very strong one. Moreover, it appears that in a complete partial metric space the condition of non-emptiness of a bottom set is equivalent to the existence of a constant mapping satisfying the condition \eqref{eqn:1} of Theorem \ref{thm:1}.

For convenience of the reader in the next section we collect some basic definitions and facts which are necessary to understand the further part of the paper. 
\section{Preliminaries}

In what follows $\bn$ denotes the set of positive integers and $\bn_0:=\bn\cup\{0\}$. For a set $A\subset X$, where $X$ is a topological space, $\cl A$ denotes the closure of $A$. For a function $f:X\to Y$, $X,\,Y$ are some sets, and a subset $A\subset X$ by $f|_A:A\to Y$ we denote the restriction of the function $f$ to the set $A$: $f|_A(x):=f(x),\, x\in A$.

Let $U$ be a nonempty set. A function $p:U\times U\to \br_+$ is called a partial metric on $U$ if for $x,y,z\in U$ the following conditions are satisfied
\begin{enumerate}[label={\textup{(\arabic*)}},ref=\textup{.(\arabic*)}]
\item $x=y\Leftrightarrow p(x,x)=p(x,y)=p(y,y)$;
\item $p(x,x)\leq p(y,x)$;
\item  $p(x,y)=p(y,x)$;
\item $p(x,y)\leq p(x,z)+p(z,y)-p(z,z)$.
\end{enumerate}
The pair $(U,p)$ is called a partial metric space.

For a given partial metric space $(U,p)$ we define (cf. \cite{IPR})
$$\rho_p:=\inf\{p(x,x):\, x\in U\},\quad U_p:=\{x\in U:\, p(x,x)=\rho_p\},$$
and call the latter the \emph{bottom set} of the space. Let us notice that $\rho_p$ is well defined and it is possible that $U_p=\emptyset$.

For a partial metric space $(U,p)$ we define mappings $p^m,\,\ov{p}:U\times U\rightarrow \mathbb{R}_+$ by
\[
p^m(x,y):=2p(x,y)-p(x,x)-p(y,y),
\]
\[
\ov{p}(x,y):=p(x,y)-\rho_p,
\]
for any $x,y\in U$. It is well-known that $p^m$ is a metric on $U$. It is obvious that $\ov{p}$ is a partial metric on $U$, but, in addition, the restriction of $\ov{p}$ to the set $U_p\times U_p$ is a metric on $U_p$. It is also clear that $U_{\ov{p}}=U_{p}$ and $\rho_{\ov{p}}=0$.

A sequence $(x_n)_{n\in \mathbb{N}}$ of elements of a partial metric space $(U,p)$ is said to \emph{converge to an element $x\in U$} if
\[
\lim_{n\rightarrow \infty}p(x_n,x)=p(x,x).
\]

A sequence $(x_n)_{n\in \mathbb{N}}$ in $(U,p)$, is said to \emph{properly converge} to $x\in U$, if it converges to $x$ and $\lim_{n\to \infty}p(x_n,x_n)=p(x,x)$. In this case, if $\lim_{n\to \infty}p(x_n,x)=p(x,x)$ and $\lim_{n\to \infty}p(x_n,x_n)=p(x,x)$, then $\lim_{m,n\to \infty}p(x_m,x_n)=p(x,x)$.

A sequence $(x_n)_{n\in \mathbb{N}}$ of elements of a partial metric space $(U,p)$ is said to be a \emph{Cauchy sequence} if there exists (and is finite)
\[
\lim_{n,m\rightarrow\infty}p(x_n,x_m).
\]
A partial metric space $(U,p)$ is said to be \emph{complete} if every Cauchy sequence $(x_n)_{n\in \mathbb{N}}$ in $(U,p)$ of elements of $U$ is properly convergent.

For any $a\geq 0$, a sequence $(x_n)_{n\in \mathbb{N}}$ of elements of a partial metric space $(U,p)$ is said to be a $a$-\emph{Cauchy sequence} if
\[
\lim_{n,m\rightarrow\infty}p(x_n,x_m)=a.
\]
Similarly, a partial metric space $(U,p)$ is said to be $a$-\emph{complete} if every $a$-Cauchy sequence $(x_n)_{n\in \mathbb{N}}$ in $(U,p)$ of elements of $U$ is properly convergent to some element $x\in U$. Clearly, in that case, $p(x,x)=a$.

Let $(U,p)$ be a partial metric space. The open ball centered at $x\in U$ with radius $\varepsilon>0$ is defined as
\[
B(x,\varepsilon):=\{y :\, p(x,y)<p(x,x)+\varepsilon,\ y\in U\}.
\]
The set of all open balls of a partial metric space $(U,p)$ is the basis of a topology of $U$, denoted by $\mathcal{T}[p]$. The topology on $U$ which is generated by the metric $p^m$ is denoted by $\mathcal{T}[p^m]$. It can be shown that the notions of convergence in topology $\mathcal{T}[p]$ ($\mathcal{T}[p^m]$) and in partial metric space $(U,p)$ (metric space $(U,p^m)$) are equivalent. Matthews \cite{M} proved that $\mathcal{T}[p]\subset \mathcal{T}[p^m]$. It is well known that the topological space $(U,\mathcal{T}[p])$ is a $T_0$ first countable space \cite{HWZ} and, in particular, it is a sequential space \cite[pp. 53--54]{ENGEL}. Let us recall that a topological space is said to be $T_1$ space if for any points $x,y$, $x\neq y$, of that space each of them possesses an open neighborhood not containing the other point. In general, partial metric spaces are not $T_1$ spaces. 

Given a partial metric space $(U,p)$, we define the mapping $D:U\times U\rightarrow \mathbb{R}_+$ by
 \begin{eqnarray*}
D(x,y) =\left\{
\begin{aligned}
  p(x,y), \quad  \mbox{if}  \  &&x\not=y, \\
  0, \quad \mbox{if} \  &&x=y.
\end{aligned}
\right.
\end{eqnarray*}
$D$ is a metric on $U$, $\mathcal{T}[p^m]\subset \mathcal{T}[D]$ and the metric space $(U,D)$ is complete if and only if $(U,p)$ is $0-$complete \cite{HRS}.

In connection with the open problem {\sl Question 8.7} from the paper \cite{HWZ} let us notice that a proper partial metric defined on a linear space may not be simultaneously translation invariant and absolutely homogeneous while a metric, defined on such a space, may be. Indeed, if $p$ is a partial metric on a linear space $X$ and $p$ is translation invariant and homogeneous, then it is a metric on $X$. Since $p(k x,k y)=|k|p(x,y)$, $p(0,0)=0$. This implies that $p(x,x)=0$ and $p(x,z)\leq p(x,y)+p(y,z)$. Moreover, $p(x,y)=0$ gives $p(x,x)=p(y,y)=p(x,y)$, so $x=y$. Hence $(X,p)$ is a metric space. Hence, if a partial metric space is a proper partial metric space (that is, it is not a metric space), then it cannot be both translation invariant and absolutely homogeneous.
Let us notice that instead of requiring absolute homogeneity it is enough to impose homogeneity for some positive constant $k$, that is, $p(kx,ky)=kp(x,y),\, x,y\in X$.

\section{Remarks on completeness of a partial metric space}
The following three characterizations of convergence modes in partial metric spaces are well-known (see \cite{ASS,CKT} or \cite[Proposition 5.24]{CO}):
\begin{prop}\label{thm:3:1} Let $(U,p)$ be a partial metric space. 
\begin{enumerate} [label={\textup{(\alph*)}},ref=\textup{\alph*}]
\item\label{thm:3:1:a} A sequence $(x_n)_{n\in \mathbb{N}}$ of elements of $U$ is properly convergent to $x\in U$ if and only if $(x_n)_{n\in \mathbb{N}}$ converges to $x$ with respect to the topology $\mathcal{T}[p^m]$.
\item\label{thm:3:1:b}
Let $(U,p)$ be a partial metric space. A sequence $(x_n)_{n\in \mathbb{N}}$ of elements of $U$ is a Cauchy sequence if and only if $(x_n)_{n\in \mathbb{N}}$ is a Cauchy sequence with respect to the topology $\mathcal{T}[p^m]$.
\item\label{thm:3:1:c}
Let $(U,p)$ be a partial metric space. $(U,p)$ is complete if and only if it is complete with respect to the topology $\mathcal{T}[p^m]$.
\end{enumerate}
\end{prop}

We know that if a partial metric space $(U,p)$ is complete then it is $0$-complete. The converse is not true.
\begin{exa}
Let $U:=(0,1)$ and for all $x,y\in U$ we define $p(x,y)$ by the formula
$$
p(x,y):=1+\max\{x,y\}
$$
Then $(U,p)$ is a partial metric space, $\rho_p=1$, and $p^m(x,y)=|x-y|$. It follows that there does not exist any 0-Cauchy sequence in $(U,p)$. Thus $(U,p)$ is $0$-complete. It is obvious that $(U,p^m)$ is not complete, so $(U,p)$ is not complete.
\end{exa}
Let us notice that the partial metric space $(U,p)$ in the above example is not metrizable because it is not a Hausdorff space. The example might appear to be a bit artificial because the reason for which the space $(U,p)$ is $0$-complete is that the thickness of the space is positive ($\rho_p=1$) - this trivially entails that there is no $0$-Cauchy sequence. However, this method of constructing (counter)examples is formally correct and seems to be useful (see Example \ref{ex:PMSnotTB}).

Let $(U,p)$ be a partial metric space. If a sequence $(x_n)_{n\in \mathbb{N}}$ of elements of $U$ is convergent to $x\in U$, then we can not deduce that $(x_n)_{n\in \mathbb{N}}$ is a Cauchy sequence. To show this, let us consider the following

\begin{exa}
Let $A:=\{a,b,c\}$ and $U:=2^A$. Let us define $p(x,y):=|x\cup y|$ for $x, y\in U$. Then $(U,p)$ is a partial metric space. Let $x:=\{a,b\}$ and
 \begin{eqnarray*}
x_n :=\left\{
\begin{array}{lll}
  \{a\}, &  \rm{if}  & n\ \mbox{is even}, \\
  \{b\}, & \rm{if} & n\ \mbox{is odd}.
\end{array}
\right.
\end{eqnarray*}
Then we have
$$
\lim_{n\to \infty}p(x_n,x)=p(x,x).
$$
Thus the sequence $(x_n)_{n\in \mathbb{N}}$ is convergent to $x\in U$. However,
 \begin{eqnarray*}
p(x_n,x_m) =\left\{
\begin{array}{lll}
  1, &  \rm{if}  & n+m\ \mbox{is even}, \\
  2, & \rm{if}  & n+m\  \mbox{is odd},
\end{array}
\right.
\end{eqnarray*}
so $(x_n)_{n\in \mathbb{N}}$ is not a Cauchy sequence.
\end{exa}
It appears that proper convergence of a sequence of elements of $U$ implies that it is a Cauchy sequence. More precisely, we have 
\begin{prop}[cf. \cite{CO} Proposition 5.12]
Let $(U,p)$ be a partial metric space. A sequence $(x_n)_{n\in \mathbb{N}}$ of elements of $U$ is properly convergent to $x\in U$, then $(x_n)_{n\in \mathbb{N}}$ is a Cauchy sequence.
\end{prop}

As we know $\mathcal{T}[p]\subset \mathcal{T}[p^m]\subset \mathcal{T}[D]$ and it turns out that a contractive mapping (that is, a mapping $T$ satisfying the condition $p(T(x),T(y))\leq \alpha p(x,y),\, x,y\in U,$ for some $0<\alpha<1$) may not be continuous with respect to the topology $\mathcal{T}[p]$ (nevertheless, it must be continuous with respect to the topology $\mathcal{T}[D]$).

To illustrate this fact, let us consider the following 
\begin{exa}\label{ex:rw}
Let $U:=\{-7,-6,-5\}\cup[0,+\infty)$ and define a partial metric $p$ on $U$ by $p(x,y):=\frac{d(x,y)+f(x)+f(y)}{2}$, $x,y\in U$, where
$$f(x):=\left\{\begin{array}{ll}3+x,&\text{ if } x\in [0,+\infty),\\0,&\text{ if }x=-5,\\1,&\text{ if }x\in \{-7,-6\},\end{array}\right.$$
$x\in U$, and $d(x,y)=|x-y|,\, x,y\in U$. That $p$ is a partial metric on $U$ comes from the fact that $f(x)\geq 0,\,|f(x)-f(y)|\leq |x-y|$, $x,y\in U$, and $d$ is a metric on $U$ (see \cite{AA}). The partial metric space $(U,p)$ is complete \cite[p.194]{M}, $\rho_p=0$ and $U_p=\{-5\}$.
Let, for $x\in U$
$$T(x)=\left\{\begin{array}{ll}-5,&\text{ if }x\in \{-7,-6,-5, 0\},\\-6,&\text{ if } x>0\text{ and } x\notin\{\frac{1}{2q}:\, q\in \bn\},\\ -7,&x\in\{\frac{1}{2q}:\, q\in \bn\}.\end{array}\right.$$
Observe that $p(x,y)\geq 3$ if $x\geq 0$ or $y\geq 0$. Further, $p(T(x),T(y))\leq 2,\, x,y\in U$, and $p(T(x),T(y))=0$, $x,y\in \{-7,-6,-5,0\}$. So, $T$ is a contraction: for $x,y\in U$, $$p(T(x),T(y))\leq \frac{2}{3}p(x,y).$$
Let $x_n:=\frac{1}{n},\, n\in \bn$. Then, for $x\geq 0,$ $\lim_{n\to \infty}p(x_n,x)=\lim_{n\to \infty}\frac{|x-\frac{1}{n}|+3+x+3+\frac{1}{n}}{2}=3+x=p(x,x)$, so $(x_n)_{n\in \bn}$ converges to each $x\geq 0$. If $x\in\{-7,-6,-5\}$, then $p(x_n,x)\neq p(x,x)$, from which we conclude that $(x_n)_{n\in \bn}$ does not converge to any $x\in\{-7,-6,-5\}$. Since $T(x_n)=-7,$ if $n$ is even, and $T(x_n)=-6,$ if $n$ is odd, and $T(x)\in \{-7,-6,-5\}$ for $x\geq 0$, and $p(T(x),T(x))\neq p(-7,T(x))$ or $p(T(x),T(x))\neq p(-6,T(x))$ for $x\geq 0$, we get that the sequence $(T(x_n))_{n\in \bn}$ does not converge to $T(x)$ for any fixed $x\geq 0$. Thus $T$ is not continuous at any $x\geq 0$.
Finally, let us also note that $x=-5$ is the only fixed point of $T$ and that the space $(U,p)$ is not a Hausdorff space.
\end{exa}

At the end of this section let us recall the following
\begin{thm}[\cite{H}]
A metric space $(X,d)$ is complete if and only if for every nonempty closed subset $Y\subset X$, every contraction on $Y$ has a fixed point in $Y$.
\end{thm}

Therefore, as a corollary, we get the following
\begin{thm}
A partial metric space $(U,p)$ is 0-complete if and only if for every nonempty subset $Y\subset X$ which is closed with respect to the topology $\mathcal{T}[D]$, every contraction on $Y$ has a fixed point in $Y$.
\end{thm}

\section{Compactness of a partial metric space}
Let us recall 
\begin{defin}
Let $(X,\mathcal T)$ be a topological space. The space $(X,\mathcal T)$ is said to be compact if for any open cover of $X$ the cover has a finite subcover of $X$.
\end{defin}
Remark that above it is not assumed that the space $X$ satisfies some form of separation axiom. However, it is clear that any partial metric space $(U,p)$ is a $T_0$-space, that is, for any two different points $x,y\in U$, there is an open set $V$ such that either $x\in V$ and $y\notin V$, or inversely. In general, a partial metric space may not be a Hausdorff space. 
\begin{defin}
Let $(X,\mathcal T)$ be a topological space. The space $(X,\mathcal T)$ is said to be sequentially compact if any sequence of elements of $X$ possesses a convergent subsequence.
\end{defin}
In general topological space setting compactness is neither necessary nor sufficient for sequential compactness. In the case of metrizable topological spaces these notions are equivalent. 
\begin{thm}\label{ComSeqCom}
If $(U,p)$ is a compact partial metric space, then it is sequentially compact.
\end{thm}
\begin{proof} Assume that $(U,p)$ is not sequentially compact, that is, there exists a sequence $(x_n)_{n\in \mathbb{N}}$ of elements of the compact partial metric space $(U,p)$ which has no convergent subsequence. It follows that the sequence $(x_n)_{n\in \mathbb{N}}$ does not contain constant subsequences. So we may assume, extracting a subsequence if necessary, that $x_n\neq x_k$ for $n\neq k$. Let $W:=\{x_1,x_{2},\ldots\}$. Since there is no convergent subsequence of $(x_n)_{n\in \bn}$, for any $x\in U$ there exists $r_x>0$ such that $B(x, r_x)$ contains at most a finite number of elements of $W$. The family $V_x:=B(x,r_x),\, x\in U,$ is an open cover of $U$. By compactness of $U$ there exist a finite number of points $y_i\in U$, $i=1,\ldots,m$, with $U=V_{y_1}\cup\ldots\cup V_{y_m}$, but this implies that there are at most a finite number of different values taken on by the sequence $(x_n)_{n\in \bn}$ which contradicts that the set $W$ is infinite. The proof is complete.
\end{proof}
The following lemma is crucial for further considerations.
\begin{lm}\label{SeqCompDiag}
Let $(U,p)$ be a partial metric. If the space $(U,\mathcal T[p])$ is $T_1$ and sequentially compact, then the diagonal of $U$, $\Delta:=\{(x,x):\, x\in U\}$, is a $G_\delta$-set in $U\times U$, that is, $\Delta=\bigcap_{n\in \bn} D_n$, for some family of open sets $D_n\subset U\times U,\, n\in \bn$.
\end{lm}
\begin{proof}
Let $D_n:= \bigcup_{x\in U}B(x,1/n)\times B(x,1/n),\,n\in \bn$. It is obvious that the sets $D_n,\, n\in \bn,$ are open in $U\times U$ and $\Delta\subset\bigcap_{n\in \bn} D_n$. Suppose that $(x,y)\in \bigcap_{n\in \bn} D_n$. Hence, for each $n\in \bn$, there is $x_n\in U$ such that $(x,y)\in B(x_n,1/n)\times B(x_n,1/n)$. 
Abusing notation a bit we may assume that $\lim_{n\to \infty}p(x_n,\ov{x})=p(\ov{x},\ov{x})$ for some $\ov{x}\in U$. Now, for any $\varepsilon>0$ and large $n$, we have $p(x,\ov{x})\leq p(x,x_n)+p(x_n,\ov{x})-p(x_n,x_n)<p(x_n,\ov{x})+1/n<p(\ov{x},\ov{x})+2\varepsilon$ and, similarly, $p(y,\ov{x})<p(\ov{x},\ov{x})+2\varepsilon$. Hence, $x,y\in B(\ov{x},2 \varepsilon)$ for $\varepsilon>0$. This, due to the fact that $U$ is a $T_1$ space, implies that $x=\ov{x}=y$. Thus $\Delta=\bigcap_{n\in \bn} D_n$.
\end{proof}

If we add the assumption that the partial metric space under consideration is Hausdorff we obtain the following result.
\begin{thm}\label{SeqComHausMetr}
Let $(U,p)$ be a Hausdorff partial metric space. If the space $(U,\mathcal{T}[p])$ is sequentially compact, then it is metrizable.
\end{thm}
\begin{proof}
It is clear that $(U,\mathcal T[p])$ is a $T_1$ space as it is a Hausdorff topological space. By Lemma \ref{SeqCompDiag}, the diagonal $\Delta$ of $U$ is a $G_\delta$--set in $U\times U$. Moreover, since $(U,\mathcal{T}[p])$ is a sequentially compact space, it is a countably compact space \cite[p. 162]{K}. 
By \cite[Corollary 2.A]{CH}, the space $U$ is metrizable.
\end{proof}
By Theorems \ref{ComSeqCom} and \ref{SeqComHausMetr} we additionally have
\begin{cor}\label{H:1} Let $(U,p)$ be a Hausdorff partial metric space. If the space $(U,\mathcal{T}[p])$ is compact, then it is metrizable.
\end{cor}
\begin{cor}\label{H:2a} Let $(U,p)$ be a compact partial metric space. The space $(U,\mathcal{T}[p])$ is metrizable if and only if it is a Hausdorff space.
\end{cor}
We are now in position to prove that the sequential compactness of a partial metric space implies its compactness without assuming that the space is Hausdorff. To this end we need 
\begin{lm}\label{lm:bound}
Let $(U,p)$ be a sequentially compact partial metric space. Then $U$ is bounded, that is, $\sup\{p(x,y):\, x,y\in U\}<+\infty$.
\end{lm}
\begin{proof}
Suppose that $x_n,y_n\in U,\,n\in \bn$, satisfy $\lim_{n\to\infty}p(x_n,y_n)=+\infty$. Without loss of generality we may assume that $\lim_{n\to\infty}p(x_n,x)=p(x,x)$ and $\lim_{n\to\infty}p(y_n,y)=p(y,y)$ for some $x,y\in U$. This gives, for $n\in \bn$, $p(x_n,y_n)\leq p(x_n,x)+p(x,y)+p(y,y_n)-p(x,x)-p(y,y)$ and taking the limit $n\to\infty$ we get $+\infty=\lim_{n\to\infty}p(x_n,y_n)\leq p(x,y)$ which is impossible. The claim follows.

\end{proof}
\begin{thrm}\label{thm:seqcompcomp}
If $(U,p)$ is a sequentially compact partial metric space, then it is compact.
\end{thrm}
\begin{proof} Let us define the binary relation $\succeq$ on $U$ as follows: for $x,y\in U$ $$x\succeq y \Leftrightarrow \forall_{\varepsilon>0}\, y\in B(x,\varepsilon).$$ 
The relation is reflexive, transitive, and antisymmetric (cf. Definition 3.3 and Theorem 3.4 in \cite{M}). 

Observe that $x\succeq y$, $x\neq y$, imply $p(x,y)=p(x,x)>p(y,y)$, $x,y\in U$.

Let us now prove that for each $x\in U$ there is an element $\hat{x}\in U$ with $\hat{x}\succeq x$ and such that there is no $y\in U, \, y\neq \hat{x},$ satisfying $y\succeq \hat{x}$. Let $W_x:=\{y\in U:\, y\succeq x\}$. Since $W_x\subset U$ and $U$ is a sequentially compact space, Lemma \ref{lm:bound} implies that $p_x:=\sup\{p(y,y):\, y\in W_x\}<+\infty$. By the sequential compactness of $U$ there exists a sequence $(x_n)_{n\in \bn}$, $x_n\in W_x,\, n\in \bn,$ converging to some $\hat{x}\in U$ with $\lim_{n\to\infty}p(x_n,x_n)=p_x$. By the convergence of $(x_n)_{n\in\bn}$ to $\hat{x}$, $\lim_{n\to\infty} p(x_n,\hat{x})=p(\hat{x},\hat{x})$ and, due to the inequality $p(x_n,x_n)\leq p(x_n,\hat{x})$, we see that $p_x\leq p(\hat{x},\hat{x})$. Again by the convergence, for any $\varepsilon>0$, we have, for large $n$, $x_n\in B(\hat{x},\, \varepsilon)$. From this observation we deduce that $x\in B(x_n,1/n)\subset B(\hat{x},\varepsilon)$ for large $n$. Hence, $\hat{x}\succeq x$, $\hat{x}\in W_x$ and $p(\hat{x},\hat{x})=p_x$. Now, if $y\neq\hat{x}$ and $y\succeq \hat{x}$, then $p(\hat{x},\hat{x})<p(\hat{x},y)=p(y,y)$, but $\hat{x}\succeq x$ gives $y\succeq x$ and, subsequently, $y\in W_x$ and $p_x<p(y,y)$ which contradicts the definition of $p_x$. Thus, for each $x\in U$, there exists (not necessarily unique) $\hat{x}\in U$ such that 
\begin{equation}\label{eq:relation}\hat{x}\succeq x \text{ and for no }y\in U,\, y\neq \hat{x}, \text{ it holds } y\succeq \hat{x}.
\end{equation}

Let now $$\widehat{U}:=\{\hat{x}\in U:\,\hat{x}\text{ meets the condition \pref{eq:relation} for some }x\in U\}.$$ By the property \pref{eq:relation} we obtain that, for each $\varepsilon>0$, the collection of balls $B(\hat{x},\varepsilon),\, \hat{x}\in \widehat{U}$, is an open cover of $U$. Let $\hat{p}:=p|_{\widehat{U}}$. It is clear that the pair $(\widehat{U},\hat{p})$ is a partial metric space with the natural (subspace) topology. By the definition of elements of $\widehat U$ it follows that if $\hat{x},\hat{y}\in \widehat U\, (\subset U)$ and $\hat x\neq \hat y$, then there exists $\varepsilon>0$ for which $\hat x\notin B(\hat y,\varepsilon)$ and $\hat y\notin B(\hat x,\varepsilon)$. Hence, the space $(\widehat U,\mathcal T[\hat p])$ is a $T_1$ topological space. We shall show that the space $\widehat{U}$ is a sequentially compact partial metric space. Let $(x_n)_{n\in \bn}\in \widehat{U}$ be a sequence of elements of $\widehat{U}$. Since $\widehat{U}\subset U$ and $U$ is a sequentially compact space, we may assume without loss of generality that $\lim_{n\to\infty}p(x_n,x)=p(x,x)$ for some $x\in U$. 
But $x\in U$, so there exists $\hat{x}\in \widehat{U}$ for which the condition \pref{eq:relation} is satisfied. Notice that, for each $\varepsilon>0$, $x\in B(\hat{x},\varepsilon)$ and, for any positive $\delta <p(\hat{x},\hat{x})+\varepsilon-p(\hat{x},x)$ and large $n$, we have $x_n\in B(x,\delta)\subset B(\hat{x},\varepsilon)$ which implies $\lim_{n\to\infty}\hat{p}(x_n,\hat{x})=\hat{p}(\hat{x},\hat{x})$. We have just proved that any sequence of elements of $\widehat{U}$ has a subsequence that converges to some element of $\widehat{U}$. 
Thus the space $(\widehat{U},\hat{p})$ is a sequentially compact partial metric space which is also $T_1$ as a topological space. Thus, by Lemma \ref{SeqCompDiag}, the diagonal of the space is a $G_\delta$--set. Since sequential compactness implies countable compactness, then, by \cite[Corollary 2.A]{CH}, the topological space $(\widehat{U},\mathcal{T}[\hat{p}])$ is compact. 

Let now $V_\alpha,\, \alpha\in I,$ be any open cover of $U$. Since $\widehat{U}\subset U$, for each $x\in \widehat{U}$, there is $\alpha_x\in I$ with $x\in V_{\alpha_x}$. By the definition of elements of the set $\widehat{U}$ the collection $V_{\alpha_x},\, x\in \widehat{U},$ openly covers both $\widehat{U}$ and $U$. Since the space $(\widehat{U},\mathcal{T}[\hat{p}])$ is compact and it is endowed with the natural subspace topology of $(U,\mathcal{T}[p])$, there are a finite number of points $x_1,\ldots,x_n,\, n\in \bn$, with $\widehat{U}\subset V_{x_1}\cup\ldots\cup V_{x_n}$. But this implies that $U\subset V_{x_1}\cup\ldots\cup V_{x_n}$.
 
\end{proof}

As we know, a metric space is compact if and only if it is totally bounded and complete. However, one can indicate a compact partial metric space which is not complete.

\begin{exa}\label{ex:4:4}
Let $U:=(0,1]$ and, for all $x,y\in U$, $p(x,y):=\max\{x,y\}$. Then $(U,p)$ is a compact partial metric space which is not complete.

Indeed, for any open cover $\{O_\lambda\}_{\lambda\in \Lambda}$ of $U$, there exist $\lambda_0\in \Lambda$ and $\varepsilon_0>0$ such that
$$
1\in \{y\in U:\, p(1,y)<p(1,1)+\varepsilon_0\}\subset O_{\lambda_0}.
$$
For any $y\in U$, we also have $p(1,y)=1$, so $U\subset \{y:\, p(1,y)<p(1,1)+\varepsilon_0\}\subset O_{\lambda_0}$. Since $\{O_\lambda\}_{\lambda\in \Lambda}$ is arbitrary, it follows that $(U,p)$ is compact.

Because $p^m(x,y)=2p(x,y)-p(x,x)-p(y,y)=|x-y|$, it is obvious that $(U,p^m)$ is not complete, so $(U,p)$ is not complete.
\end{exa}

\begin{defin}\label{df:4:5}
Let $(U,p)$ be a partial metric space. We call $(U,p)$ $p$-totally bounded if, for any $\varepsilon>0$, there exist $x_1,x_2,\ldots, x_n\in U,\, n\in \bn,$ such that the collection of open balls $B(x_i,\varepsilon)$, $i=1,2,\ldots,n$, is a cover of $U$.
\end{defin}
If $(U,p)$ is a metric space which is $p$-totally bounded, then any subset of $U$ is also $p$-totally bounded. This property is not preserved even for compact partial metric spaces (cf. Question 8.7 in \cite{HWZ}).
\begin{prop}
In the class of compact partial metric spaces, $p$-total boundedness is non--hereditary, that is, there exists a compact partial metric space $(U,p)$ which is $p$-totally bounded and there is a subspace $(X,p|_X)$ of $(U,p)$, $X\subset U$, which is not $p|_X$-totally bounded.
\end{prop}
\begin{proof}
Let $(X,d)$ be any non-compact metric space with the discrete metric $d$: $d(x,y):=1,\, x\neq y,\, d(x,x):=0,\, x,y\in X$. Let $a:=\{X\}$, so that $a\notin X$. Define $U:=X\cup \{a\}$ and let $p:U\to \br_+$ be given by $p(x,y):=d(x,y),\,x,y\in X$, $p(x,a):=p(a,x):=2,\,x\in X$, $p(a,a):=2$. One can check that the pair $(U,p)$ is a partial metric space. The space $(U,p)$ is compact because for any $\varepsilon\in (0,1)$ we have $U\subset B(a,\varepsilon)$ and $a\notin B(x,\varepsilon),\, x\in X$. Observe that the subspace $X\subset U$ is not $p|_X$-totally bounded. Let us also notice that the only Cauchy sequences in the space $(U,p)$ (or $(X,p|_X)$) are those which are eventually constant.
\end{proof}

We know that if $(x_n)_{n\in \mathbb{N}}$ is a sequence of elements of a metric space $(X,d)$ and $d(x_n,x_m)\geq \varepsilon_0$ for $n,m \in \mathbb N, n\neq m,$ for some $\varepsilon_0>0$, then $(x_n)_{n\in \mathbb{N}}$ does not have a convergent subsequence. However, it is not true for partial metric spaces, which can be illustrated by the following

\begin{exa}\label{ex:PMSnotTB}
Let $U=\{0\}\cup \mathbb N $ and let us define $p(n,m)$ by the formula
\begin{eqnarray*}
p(n,m) =\left\{
\begin{array}{lll}
  1+\frac{1}{n}+\frac{1}{m}, &  \rm{if}  \ &n,m\in\mathbb N,\ n\neq m,  \\
  1, & \rm{if} \  &n=m\in U,\\
	1+\frac1n, & \rm{if} \ & m=0,\ n\in\mathbb N,\\
  1+\frac1m, & \rm{if} \ & n=0,\ m\in\mathbb N.
\end{array}
\right.
\end{eqnarray*}
We have $p(n,m)>1$ for all $n,m\in U$ with $n\neq m$, but $\lim_{n\to \infty}p(n,0)=p(0,0)$, that is $(n)_{n\in \mathbb N}$ converges to $0\in U$, as $n\to \infty$.
\end{exa}

\section{$U_p$ and fixed point theorems}

The initial point of this section is the following 
\begin{thm}[cf. \cite{IPR}]\label{thm:1}
Let $(U,p)$ be a complete partial metric space. Let $T:U\to U$ be a mapping satisfying for all $x,y\in U$ the following condition
\begin{equation}\label{eqn:1}p(T(x),T(y))\leq \max\{\alpha p(x,y),p(x,x),p(y,y)\},\end{equation}
where $\alpha \in [0,1)$ is fixed. Then there exists a unique point $\ov{x}\in U_p$ for which $T(\ov{x})=\ov{x}$. Moreover, $\lim_{n\to\infty}p(T^n(x),\ov{x})=p(\ov{x},\ov{x})=\lim_{m,n\to\infty}p(T^n(x),T^m(x))$ for any $x \in U_p$, and $T(U_p)\subset U_p$, and the mapping $T$ is continuous at any $x\in U_p$.
\end{thm}

\begin{rem} 
Let us notice that in comparison to the original statement of Theorem \ref{thm:1} in \cite{IPR} we added two more observations, namely $T(U_p)\subset U_p$, and that the mapping $T$ is continuous at any $x\in U_p$. In fact, by \pref{eqn:1}, $\rho_p\leq p(T(x),T(x))\leq p(x,x)=\rho_p$ for $x\in U_p$, and we see that $T(x)\in U_p$. Further, if $\lim_{n\to\infty}p(x_n,x)=p(x,x),\, x_n\in U,\, n\in \bn,\, x\in U_p$, then $\rho_p\leq \lim_{n\to\infty} p(T(x_n),T(x))\leq \lim_{n\to\infty}\max\{\alpha p(x_n,x),p(x_n,x_n),p(x,x)\}=\max\{\alpha p(x,x),p(x,x)\}=\rho_p=p(T(x),T(x))$ and continuity follows. 

Let us also observe that, by Theorem \ref{thm:1}, if $U_p=\emptyset$, then there is no mapping $T:U\to U$ for which condition (\ref{eqn:1}) is satisfied. It is also evident that $U_p\neq\emptyset$ is independent of whether $(U,p)$ is complete or not - see Examples \ref{ex:1} and \ref{ex:2} below.
\end{rem}

The following lemma provides simple observations on the set $U_p$.
\begin{lm}\label{lm:2}
Let $(U,p)$ be a $\rho_p$-complete partial metric space and $x_n\in U$, $n\in \bn$.
\begin{enumerate}[label={\textup{(\arabic*)}},ref=\textup{\arabic*}]
\item\label{lm:2:1} if $\lim_{n\to\infty}p(x_n,x)=p(x,x)$ and $x\in U_p$, then the convergence is proper; {here the assumption of completeness can be discarded.}
\item\label{lm:2:2} $U_p\neq\emptyset$ if and only if there exists a $\rho_p$-Cauchy sequence of elements of  $U$.
\item\label{lm:2:3} $(U_p,\ov{p})$ is a complete metric space.
\end{enumerate}
\end{lm}
\begin{proof}
\pref{lm:2:1} If $\lim_{n\to\infty}p(x_n,x)=p(x,x)=\rho_p$, then $p(x_n,x_n)\leq p(x_n,x)\leq \rho_p+\varepsilon$, for any fixed $\varepsilon>0$ and sufficiently large $n$. This shows that $\lim_{n\to\infty}p(x_n,x_n)=\rho_p$, and the claim follows.

\pref{lm:2:2} If $x\in U_p$, then the constant sequence $x_n:=x,\,n\in \bn$, is $\rho_p$-Cauchy. The other implication is obvious.

\pref{lm:2:3} Let $x_n\in U_p,\, n\in \bn$, be a Cauchy sequence, that is, $\lim_{n,m\to\infty}\ov{p}(x_n,x_m)=0$. It holds that $\lim_{n,m\to\infty} p(x_n,x_m)=\rho_p$, and hence, there is some $x\in U$ such that $\lim_{n\to \infty}p(x_n,x)=p(x,x)$. By the very definition of proper convergence it follows that $\rho_p=\lim_{n\to\infty}p(x_n,x_n)=p(x,x)=\lim_{n\to\infty}p(x_n,x)$, which gives $x\in U_p$ and $\lim_{n\to\infty}\ov{p}(x_n,x)=0$.

\end{proof}
\begin{rmk}\label{rmk:1}\textup{
Let us observe that the assumption $U_p\neq \emptyset$ (instead of deriving it as a conclusion) significantly simplifies the proof of Theorem \ref{thm:1}. Actually, in such a case it is a simple consequence of the Banach contraction principle. First, let us observe that $T(U_p)\subset U_p$ is an immediate consequence of the condition \pref{eqn:1} of Theorem \ref{thm:1}. Further, the metric space $(U_p,\ov{p})$ is complete (Lemma \ref{lm:2}.\pref{lm:2:3}). Finally, $\ov{T}:U_p\to U_p$, $\ov{T}(x):=T(x),\, x\in U_p$, is a contraction on $U_p$ because, for any $x,y\in U_p$, we have
\begin{multline*}\left[\,p(T(x),T(y))\leq \max\{\alpha p(x,y), p(x,x),p(y,y)\}\Leftrightarrow\right. \\p(T(x),T(y))-\rho_p\leq \max\{\alpha p(x,y)-\rho_p, 0,0\}\left.\right]\Rightarrow \ov{p}(\ov{T}(x),\ov{T}(y))\leq \alpha\ov{p}(x,y).\end{multline*}
}
\end{rmk}
Notice that a similar observation can be made on Matthews' generalization of the Banach contraction principle to partial metric spaces \cite{M}.

Remark \ref{rmk:1} raises the following question: should the fact $U_p\neq\emptyset$ be an assumption or a consequence in theorems like Theorem \ref{thm:1}? We shall show that none of possible answers is satisfactory - see Example \ref{ex:2} and Theorem \ref{thm:3} below. But before we pursue that issue let us show that Theorem \ref{thm:1} cannot be treated as a trivial consequence of the Banach contraction principle.
\begin{xmpl}\label{ex:0}\textup{
This is a slightly modified Example 3.1 from \cite{IPR} - the only modification is in the value of the mapping $T$ at $x=2$ (in the original example $T(2)=1$).
Let $U:=[0,1]\cup[2,3]$ and let the partial metric $p$ on $U$ be given by
$$p(x,y):=\left\{\begin{array}{cc}\max\{x,y\},& \text{ if }\{x,y\}\cap[2,3]\neq \emptyset, \\|x-y|, & \text{ if }\{x,y\}\subset [0,1].\end{array}\right.$$
Then $(U,p)$ is a complete partial metric space. Define $T:U\to U$ by
$$T(x):=\left\{\begin{array}{cc}\frac{x+1}{2}, & \text{ if }x\in [0,1],\\\frac{2+x}{2}, & \text{ if }x\in [2,3].\end{array}\right.$$
The function $T$ meets the condition \pref{eqn:1} of Theorem \ref{thm:1} with $\alpha=1/2$. Moreover, $T$ has two fixed points: $1$ and $2$. The bottom set of $(U,p)$ is $U_p=[0,1]$, and $\rho_p=0$. It should be emphasized that the fixed point $2\notin U_p$. This shows that Theorem \ref{thm:1} is essentially different from many fixed point theorems in which assumptions guarantee that fixed points are members of the bottom set of considered partial metric space.
Let us also notice that, for any $x\in U$, the sequence $(T^n(x))_{n\in \bn}$ converges to a fixed point of $T$, as $n\to\infty$; this is not true for the original version of the example as stated in \cite{IPR}.
}
\end{xmpl}

For a set $U$ and $z\in U$, by $T_z$ we denote the constant mapping $$T_z(x):=z,\, x\in U.$$

\begin{xmpl}\label{ex:1}\textup{
Let $U:=[0,1]$ and define a partial metric $p$ on $U$ by
$$ p(x,y):=\left\{\begin{array}{ll}0,&\text{ if } x=y>0,\\1,&\text{ if }x\neq y\text{ or } x=y=0,\end{array}\right.$$
for $x,y\in U$. Then $\lim_{n\to\infty} p(\frac{1}{n},0)=\lim_{n\to\infty} 1=1=p(0,0)$, so $\frac{1}{n}$ converges to $0$, although the convergence is not proper since $\lim_{n\to\infty} p(\frac{1}{n},\frac{1}{n})=\lim_{n\to\infty} 0=0$. It is obvious that $U_p=(0,1]$ and $U_p$ is not closed as a subset of the partial metric space $(U,p)$, so it is not a complete subspace of the partial metric space $(U,p)$. However, the metric space $(U_p, \ov{p})$ is complete (cf. Lemma \ref{lm:2}).
Observe that, for $z\in U$, the constant mapping $T_z$, $z\in (0,1]$, meets the condition \pref{eqn:1} of Theorem \ref{thm:1} for any $\alpha\in (0,1]$ and its only fixed point is $z$. It is also clear that the constant mapping $T_0$ is does not satisfy the condition \pref{eqn:1} for any $\alpha\in [0,1)$, since $p(T_0(1),T_0(1))=p(0,0)=1>0=\max\{\alpha p(1,1),p(1,1)\}$.
}
\end{xmpl}
\begin{xmpl}\label{ex:2}\textup{
Let $U:=\{\frac{1}{q+1}:\, q\in \bn\}\cup\{0\}$ and define a partial metric $p$ on $U$ by
$$p(x,y):=\left\{\begin{array}{ll}x,&\text{ if } x=y>0,\\1,&\text{ if }x\neq y\text{ or } x=y=0,\end{array}\right.$$
for $x,y\in U$.
Notice that $\rho_p=0$ and $p(x,x)>0$ for $x\in U$. Thus, $U_p=\emptyset$ and no mapping satisfying the condition \pref{eqn:1} of Theorem \ref{thm:1} exists.
The space $(U,p)$ is complete, because a sequence $({x_n})_{n\in \bn}\subset U$ is Cauchy if and only if it is eventually constant. Hence, the space $(U,p)$ is complete and, thus, $0$-complete as well, but there are no $0$-Cauchy sequences. This implies that no theorem that ensures the existence of a fixed point in $U_p$ can be applied to the space $(U,p)$.
For each sequence $(x_n)_{n\in \bn}$ in $U$, constant or not, we have $p(x_n,0)=p(0,0),\, n\in \bn,$ so it converges to $0$. Hence,  the space is not Hausdorff. The space $(U,p)$ is compact, since any open cover of $U$ has an element, say $V$, to which $0$ belongs to and there is an open ball $B(0,\varepsilon)$, $\varepsilon>0,$ with $B(0,\varepsilon)\subset V$. But any open ball centered at $0$ contains the set $U$.
}
\end{xmpl}
Examples \ref{ex:1} and \ref{ex:2}, together with Theorem \ref{thm:1}, suggest the following characterization of the non-emptiness of the bottom set of a complete partial metric space by constant functions.
\begin{thrm}\label{thm:3}
Let $(U,p)$ be a complete partial metric space. Then $U_p\neq \emptyset$ if and only if there is $z\in U$ such that the constant mapping $T_z$ satisfies the condition \pref{eqn:1} of Theorem \ref{thm:1} for any $\alpha \in [0,1)$. Moreover, $$U_p=\{z\in U:\, T_z\text{ satisfies the condition \pref{eqn:1} of Theorem \ref{thm:1}}\}.$$
\end{thrm}
\begin{proof} Suppose that $z\in U_p$, that is, $\rho_p=p(z,z)\leq p(x,y)$, $x,y\in U$. Then $p(T_z(x),T_z(y))=p(z,z)=\rho_p\leq \max\{\alpha p(x,y),p(x,x),p(y,y)\}$ for any $\alpha \in [0,1)$.

Assume now that there is $z\in U$ for which the constant mapping $T_z$ meets the condition \pref{eqn:1} of Theorem \ref{thm:1}. By Theorem \ref{thm:1}, the set $U_p$ is nonempty and it contains a fixed point of $T_z$. It is clear that $z$ is the only fixed point of $T_z$ and thus $z\in U_p$. (A direct proof without any reference to Theorem \ref{thm:1} is: if there is $z'\in U$: $p(z',z')<p(z,z)$, then $p(z,z)=p(T_z(z'),T_z(z'))\leq \max\{\alpha p(z',z'),p(z',z')\}=p(z',z')$ - a contradiction.)

From the above considerations we conclude that $$U_p=\{z\in U:\, T_z\text{ satisfies the condition \pref{eqn:1} of Theorem \ref{thm:1}}\}.$$
\end{proof}
In connection with Theorem \ref{thm:3}, let us consider the following
\begin{xmpl}\label{ex:3}\textup{
Let $U:=\{a,b\}$, $a\neq b$, and $p(a,a):=0,\,p(b,b):=1,\, p(a,b):=p(b,a):=2$. The function $p$ is a partial metric on $U$. Obviously, $U_p=\{a\}$, so the only constant mapping on $U$ satisfying the condition \pref{eqn:1} of Theorem \ref{thm:1} is $T_a$. Notice that $p(T_a(a),T_a(a))=p(a,a)=\max\{\alpha p(a,a),p(a,a)\}$, so $T_a$ is not a contraction. Observe also that $T_a$ is the only mapping from $U$ to $U$ that satisfies the condition \pref{eqn:1}. Indeed, if $T:U\to U$ is not constant, then $p(T(a),T(b))=2>\max\{\alpha p(a,b),p(a,a),p(b,b)\}=\max\{2\alpha ,1\}$ for any $\alpha\in [0,1)$.}
\end{xmpl}
At the end of this section let us present another fixed point theorem from the paper \cite{BR} in which there appears the assumption concerning the non-emptiness of the set $U_p$.
\begin{thm}[\cite{BR}, cf. \cite{ADA, MD}]\label{wrong}
Let $(U,p)$ be a complete partial metric space and let $T:U\rightarrow U$ be a mapping satisfying the condition, for $x,y\in U$ 
\begin{equation}\label{eq2}
\min\{p(T(x),T(y)),\cdots,p(T^k(x),T^k(y))\}\leq \frac{p(x,x)+p(y,y)}{2},
\end{equation}
where $k\in \bn$ is fixed. If $U_p\not=\emptyset$, then $T$ possesses a unique fixed point in $U$.
\end{thm}

\begin{rem}\label{rmk:q}

\begin{enumerate}[label={\textup{(\alph*)}},ref=\textup{\alph*}]
\item\label{rmk:q:1} If $(U,p)$ is a complete metric space (so a complete partial metric space, as well) and $T$ is a contraction on $U$ (that is, $p(T(x),T(y))\leq Lp(x,y),\, x,y\in U$, for some fixed $L\in [0,1)$), then $T$ satisfies condition \eqref{eq2} if and only if $T^k$ is a constant mapping.

We have $0\leq p(T^k(x),T^k(y))=\min\{p(T(x),T(y)),\ldots, p(T^k(x),T^k(y))\}\leq (p(x,x)+p(y,y))/2=0$ and therefore $p(T^k(x),T^k(y))=0,\, x,y\in U$. This implies that $T^k$ is a constant mapping.
\item\label{rmk:q:2} Let $T$ be as in Theorem \ref{wrong}. If $x=T(x)$, then $x\in U_p$ and, consequently, $U_p\neq\emptyset$.

Suppose that $x=T(x)$, but $x\notin U_p$. Thus there exists $x'\in U$ with $p(x',x')<p(x,x)$. Assume for a while that $T(x')\neq x'$. From (\ref{eq2}) we have $
\min\{p(T(x'),T(x)),\ldots, p(T^k(x'),T^k(x))\}\leq (p(x',x')+p(x,x))/2<p(x,x)$ and, for some $i\in \{1,\ldots, k\}$, $p(T^i(x'),x)=p(T^i(x'),T^i(x))<p(x,x)\leq p(T^i(x'),x)$ which cannot be true. Hence, if $p(x',x')<p(x,x)$, then $T(x')=x'$. This shows that any fixed point of $T$ belongs to the bottom set $U_p$. This fact and condition \pref{eq2} easily entail that there is at most one fixed point of $T$. We conclude that $x\in U_p$.
\item\label{rmk:q:3} Let $T$ be as in Theorem \ref{wrong}. Then $U_p\neq \emptyset$ if and only if $T$ has a unique fixed point.

This comes from Theorem \ref{wrong} and item \pref{rmk:q:2}.
\end{enumerate}
\end{rem}
\begin{rem}\label{rmk:qq}
Let $(U,p)$ be a complete partial metric space. Then
\begin{enumerate}[label={\textup{(\alph*)}},ref=\textup{\alph*}]
\item\label{rmk:qq:1} $U_p\neq\emptyset $ if and only if there is a mapping $T:U\to U$ that satisfies condition \pref{eq2} of Theorem \ref{wrong} and possesses a fixed point.

Suppose that $a\in U_p$. Define $T(x):=a$, $x\in U$ - it is easy to check that $T$ meets the condition \pref{eq2} of Theorem \ref{wrong}. The other implication is a consequence of item \pref{rmk:q:2} of Remark \ref{rmk:q}.

\item\label{rmk:qq:2} $U_p=\emptyset$ if and only if no constant mapping $T:U\to U$ satisfies condition (\ref{eq2}) of Theorem \ref{wrong}.

Suppose that $U_p=\emptyset$. Let $a\in U$ and define $T(x):=a,\, x\in U$. Obviously, there exists $x\in U$ with $p(x,x)<p(a,a)$. By the condition (\ref{eq2}) of Theorem \ref{wrong}, we get $p(a,a)=\min\{p(T(a),T(x)),\ldots, p(T^k(a),T^k(x))\}\leq (p(a,a)+p(x,x))/2$ which implies that $p(a,a)\leq p(x,x)$ - an absurd.
The reverse implication stems from item \pref{rmk:qq:1}.
\end{enumerate}
\end{rem}

\subsection*{Acknowledgements}
Ruidong Wang is grateful to the Faculty of Mathematics and Computer Science, Adam Mickiewicz University in Pozna\'{n}, for the excellent working conditions during the visit in February 2018 to August 2018. Also this research was partly supported by the Natural Science Foundation of China (Grant Nos. 11201337, 11201338, 11371201, 11301384).

\end{document}